\documentclass[12pt]{article}
\usepackage{graphicx} 
\usepackage{epstopdf}
\usepackage[utf8]{inputenc}
\usepackage{subfigure}
\usepackage{color}
\usepackage[english]{babel}
\usepackage{amsmath,amsthm}
\usepackage{amssymb}
\usepackage{amsfonts}
\usepackage{booktabs}
\usepackage{graphicx}
\usepackage{epstopdf}
\usepackage{listings}
\usepackage{xcolor}

\usepackage{graphicx}
\usepackage{subfigure}
\usepackage{tabularx}
\newtheorem{theorem}{Theorem}

\lstdefinestyle{PythonStyle}{
    language=Python,
    basicstyle=\ttfamily\small,
    keywordstyle=\color{blue},
    commentstyle=\color{gray},
    stringstyle=\color{red},
    showstringspaces=false,
    breaklines=true,
    numbers=left,
    numberstyle=\tiny\color{gray},
    frame=shadowbox,
    rulesepcolor=\color{gray},
    captionpos=b,
    escapeinside={(*@}{@*)}
}

\newtheorem{thm}{Theorem}[section]
\newtheorem{defi}{Definition}[section]

\newtheorem{lem}[thm]{Lemma}

\theoremstyle{plain}

\title{ The degeneracy and Alon-Tarsi number under $F$-sum operations}

\author { Zhiguo Li{$^*$}, Zhentao Jiao, Zeling Shao\\
	{\small School of Science, Hebei University of Technology, Tianjin 300401, China}
	\date{}
	\footnote{Corresponding author. E-mail: zhiguolee@hebut.edu.cn}
	\footnote{This work was supported in part by the National Natural Science Foundation of China (No. 12571345) and the Natural Science Foundation of Hebei Province, China (No. A2021202013).}
}

\begin{document}
	\baselineskip 0.65cm
	
	\maketitle
	
	\begin{abstract}

The Alon-Tarsi number of a graph $ G $ is the smallest $ k $ such that there exists an orientation $ D $ of $ G $ with maximum outdegree $ k - 1 $ satisfying that the number of even Eulerian subgraphs is different from the number of odd Eulerian subgraphs. The degeneracy of a graph $ G $ is the maximum value of the minimum degree over all subgraphs of $ G $. In this paper, we obtain a characterization of graphs with $AT(G)=2$ for any graph $G$, and study the Alon-Tarsi number of $F$-sum in terms of degeneracy.
		

		\noindent\textbf{2000 MR Subject Classification.} 05C15
	\end{abstract}
\section{Introduction}

In this paper, we confine our attention solely to simple and finite graphs. The chromatic number of a graph $ G $, denoted by $ \chi(G) $, is defined as the smallest positive integer $ k $ for which $ G $ admits a proper vertex coloring using exactly $ k $ colors. List coloring represents a well-established variation of vertex coloring. In the context of list coloring, a $ k $-list assignment for a graph $ G $ is a function $ L $ that assigns to each vertex $ v $ of $ G $ a set $ L(v) $ containing $ k $ proper colors. An $ L $-coloring of $ G $ refers to a coloring $ f $ of $ G $ where $ f(v) $ is an element of $ L(v) $ for every vertex $ v $. We say that $ G $ is $ L $-colorable if there exists a proper $ L $-coloring of $ G $. A graph $ G $ is said to be $ k $-choosable if it is $ L $-colorable for every possible $ k $-list assignment $ L $. The choice number of a graph $ G $, denoted by $ ch(G) $, is the smallest positive integer $ k $ for which $ G $ is $ k $-choosable.

A subdigraph $ H $ of a directed graph $ D $ is called Eulerian if the indegree $ d^-_H(v) $ of every vertex $ v $ in $ H $ is equal to its outdegree $ d^+_H(v) $. We do not assume that $ H $ is connected. $ H $ is even if the number of arcs of $ H $ is even, otherwise, it is odd. Let $ E_e(D) $ and $ E_o(D) $ denote the families of even and odd Eulerian subgraphs of $ D $, respectively. Let $ \text{diff}(D) = |E_e(D)| - |E_o(D)| $. We say that $ D $ is Alon-Tarsi if $ \text{diff}(D) \neq 0 $. If an orientation $ D $ of $ G $ yields an Alon-Tarsi digraph, then we say $ D $ is an Alon-Tarsi orientation (or an AT-orientation, for short) of $ G $.

The Alon-Tarsi number of $ G $ ($ AT(G) $, for short) is the smallest $ k $ such that there exists an orientation $ D $ of $ G $ with max outdegree $ k - 1 $ satisfying the number of Eulerian subgraphs with even arcs different from the number of Eulerian subgraphs with odd arcs. It was proposed by Alon and Tarsi$^{[1]}$, subsequently, they used algebraic methods to prove that $ \chi(G) \leq ch(G) \leq AT(G) $. The graph $ G $ is called chromatic-choosable if $ \chi(G) = ch(G) $. The graph $ G $ is called chromatic AT choosable if $ \chi(G) = AT(G) $.

Graph theory is replete with a variety of operations, including product, complement, addition, switching, subdivision, and deletion. These operations are instrumental in elucidating the properties of complex graphs by examining their simpler analogues. Within the domain of chemical graph theory, these operations are especially valuable for constructing elaborate molecular structures from basic components. Recently, a plethora of molecular structure classes have been thoroughly examined by employing these graph operations. Moreover, the $F$-sum, a notion that stems from the Cartesian product, will be explored in Section 2.

In this article, we obtained the exact value of the degeneracy of $ F $-sum Graph and derived an upper bound for its Alon-Tarsi number in Section 3. Meanwhile, in Section 3, we also obtained the Alon-Tarsi number of subdivision graph and $ S $-sum graph.

\section{Preliminaries}
The following will introduce the relevant background, definitions, lemmas, and corollaries derived from the lemmas.

In the in-depth study of chemistry and mathematics, the Wiener index has emerged as a prominent concept. After further exploration, four new operations based on the Wiener index have been proposed. We will present the relevant definitions and then proceed with further research.

\begin{defi}$^{[2]}$ For a connected graph $ G $, the definitions of four related graphs are as follows.

(1) $ S(G) $ is the \textit{subdivision graph} of $ G $, which is obtained by inserting an additional vertex into each edge of $ G $. In other words, each edge of $ G $ is replaced by a path of length 2.

(2) $ R(G) $ is the \textit{triangle parallel graph} of $ G $, which is constructed by adding a new vertex corresponding to each edge of $ G $ and connecting each new vertex to the endpoints of the corresponding edge.

(3) $ Q(G) $ is the \textit{line superposition graph} of $ G $, which is obtained by first performing $ S(G) $ and then connecting the new vertices if and only if the corresponding edges in $ G $ are adjacent. 

(4) $ T(G) $ is the \textit{total graph} of $ G $, which is constructed by first performing $ R(G) $ and then connecting the new vertices corresponding to adjacent edges in $ G $, in other words, the graph $ T(G) $ obtained by combining $ R(G) $ and $ Q(G) $(see Figure 1 for $P_4,S(P_4),R(P_4),Q(P_4),T(P_4)$).
\end{defi}

\begin{figure}[htbp]
	\centering
	\includegraphics[height=7.5cm, width=0.56\textwidth]{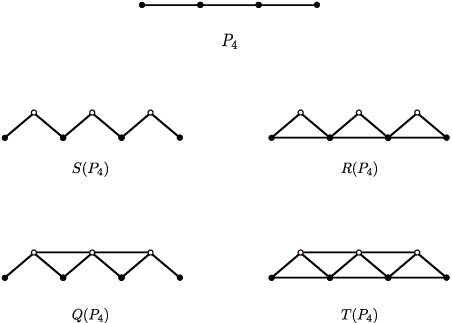}
	\caption{ $P_4$ and $S(P_4)$, $R(P_4)$, $Q(P_4)$, $T(P_4).$}
\end{figure}

\begin{defi}$^{[4]}$ The Cartesian product of graphs $ G $ and $ H $,  denoted by $ G \square H $, is the graph with vertex set $ V(G) \times V(H) $ and edges created such that $ (u, v) $ is adjacent to $ (u', v') $ if and only if either $ u = u' $ and $ vv' \in E(H) $ or $ v = v' $ and $ uu' \in E(G) $(see Figure 2 for $G = P_2,H = C_3$).
\end{defi}

\begin{figure}[htbp]
	\centering
	\includegraphics[height=5cm, width=0.57\textwidth]{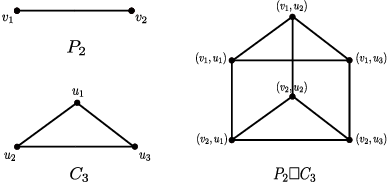}
	\caption{$P_2,C_3$ and $ P_2 \square C_3 $ .}
\end{figure}

Note that $ G \square H $ contains $ |V(G)| $ copies of $ H $ and $ |V(H)| $ copies of $ G $. There are some results concerning the Alon-Tarsi number of Cartesian product of graphs. Kaul and Mudrock proved that the Cartesian product of cycle and path has $ AT(C_k \square P_n) = 3 $ in[4]($ k $, $ n $ are integers). 

\begin{defi}$^{[5]}$ Let $ F \in \{ S, R, Q, T \} $. The $ F $-sum of $ G_1 $ and $ G_2 $, denoted by $ G_1 {+}_{F} G_2 $, is defined by $ F(G_1) \square G_2 - E^* $, where $ E^* = \{ (u, v_1)(u, v_2) \in E(F(G_1) \times G_2) : u \in V(F(G_1)) - V(G_1), v_1 v_2 \in E(G_2) \} $, i.e., $ G_1 {+}_{F} G_2 $ is a graph with the set of vertices $ V(G_1 {+}_{F} G_2) = (V(G_1) \cup E(G_1)) \times V(G_2) $ and two vertices $ (u_1, u_2) $ and $ (v_1, v_2) $ of $ G_1 {+}_{F} G_2 $ are adjacent if and only if $ [u_1 = v_1 \in V(G_1) \text{ and } u_2 v_2 \in E(G_2)] $ or $ [u_2 = v_2 \in V(G_2) \text{ and } u_1 v_1 \in E(F(G_1))] $(see Figure 3 for $G = P_4,H = P_3$).

\end{defi}

\begin{figure}[htbp]
	\centering
	\includegraphics[height=7.4cm, width=0.45\textwidth]{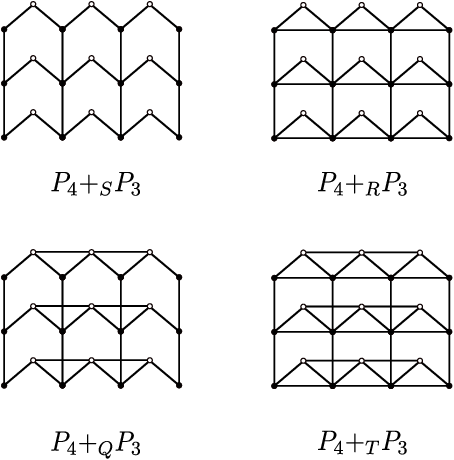}
	\caption{$ P_4 {+}_{F} P_3 $.}
\end{figure}

\begin{defi}$^{[6]}$
A graph is {$k$-degenerate} if its vertices can be successively deleted so that each deleted vertex has degree at most $k$ at the time of deletion. The {degeneracy} of a graph is the smallest integer $k$ such that the graph is $k$-degenerate, denoted as $ d $. Similarly, the degeneracy is also defined: $d = \max\{\delta(H) : H \text{ is a subgraph of } G\}$.
\end{defi}

According to Reference~[9], we introduce a general iterative method for handling $ k $-degenerate problems. Let $ V_G^0 $ denote the set of vertices in graph $ G = G^0 $ with degree at most $ k $, $ |V(G)| = m $, $ |V(H)| = n $. Construct an iterative process: $ G^i = G^{i-1} - V_G^{i-1} $, where $ V_G^{i-1} $ denotes the set of vertices in graph $ G^{i-1} $ with degree at most $ k $, for $ 1 \leq i \leq m - |V_G^0| $. The iterative process terminates when graph $ G^{s+1} $ (for $ 1 \leq s \leq m - |V_G^0| $) is an empty graph, which implies that $ G $ is a $ k $-degenerate graph. The condition for the iterative process to continue is $ |V_G^{i-1}| \geq 1 $. 

It is easy to see that a $0$-degenerate graph is an empty graph, a $1$-degenerate graph is a forest, a planar graph is a $5$-degenerate graph, and a subgraph of a $k$-degenerate graph is still a $k$-degenerate graph.

\begin{defi}$^{[6]}$ The \textit{coloring number} of a graph $G$ is defined as
\[
\text{col}(G) = 1 + \max\{\delta(H) : H \text{ is a subgraph of } G\},
\]
where $\delta(H)$ stands for the minimum degree of $H$.

\end{defi}

Alternatively, the coloring number $\text{col}(G)$ is defined as the smallest integer $l$ such that there exists a linear ordering $<$ of the vertices of $G$. Under this ordering, each vertex $v$ has at most $l-1$ neighbors $u$ satisfying $u < v$. By directing each edge $uv$ as $(u,v)$ whenever $u > v$, we achieve an acyclic orientation $D$ of $G$, where the maximum out-degree $\delta_D^+$ equals $\text{col}(G) - 1 = d$. Consequently, for any graph  G , we obtain $ AT(G) \leq d + 1 \leq k + 1 $.

\begin{defi}$^{[12]}$ In standard graph theory, $\Theta_{a,b,c}$ (where $a, b, c$ are positive integers) usually denotes a generalized $\Theta$-graph, which is formed by connecting two vertices $u$ and $v$ with three internally disjoint paths whose lengths (number of edges) are $a, b, c$, respectively(see Figure 4 for $\Theta_{2,2,2}$ and $\Theta_{2,2,4}$).
\end{defi}

\begin{figure}[htbp]
	\centering
	\includegraphics[height=3.6cm, width=0.43\textwidth]{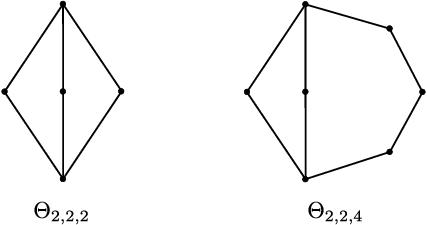}
	\caption{$\Theta_{2,2,2}$ and $\Theta_{2,2,4}$.}
\end{figure}

\begin{defi}$^{[12]}$ The core of a graph $ G $ is the subgraph obtained by successively pruning away vertices of degree 1 until no such vertices remain.
\end{defi}

\begin{lem}$^{[9]}$ The Cartesian product graph $ G \square H $ of a $ k $-degenerate graph $ G $ and an $ l $-degenerate graph $ H $ is $ (k + l) $-degenerate.
\end{lem}

\begin{lem}$^{[7]}$ If $ G $ is a bipartite graph, then every orientation $ D $ of $ G $ is an AT-orientation. Moreover,
\[
AT(G) = \max_{\substack{\text{subgraph } H \text{ of } G}} \left\lceil \frac{|E(H)|}{|V(H)|} \right\rceil + 1.
\]
\end{lem}

\begin{lem}$^{[12]}$ A graph $G$ is 2-choosable if and only if the core of $G$
 belongs to $T = \{K_1, C_{2m+2}, \\ \Theta_{2,2,2m} : m \geq 1\}$.
\end{lem}

\begin{lem} 
For a connected graph $G$ with at least one edge, $AT(G) = 2$ if and only if the core of $G$ belongs to $\{K_1, C_{2m+2}\}$.
\end{lem}

\noindent\textit{Proof:} 
We first prove the sufficiency. Since the core of $G$ belongs to $\{K_1, C_{2m+2}\}$, and by the definition of the core, the graph $G$ can only be a tree, an even cycle, or a unicycle(with the requirement that the cycle is an even cycle).

When $G$ is a tree or an even cycle, it is clear that $AT(G)=2$. When $G$ is a unicycle(with the requirement that the cycle is an even cycle), by orienting the even cycle as a directed cycle and orienting all attached trees consistently towards the cycle, we obtain an AT-orientation with maximum outdegree at most 1. Hence, $AT(G) \le 2$. Furthermore, since it is obvious that $AT(G) \ge 2$, we have $AT(G)=2$.

Now consider the necessity. Since $AT(G) = 2$, we have $ch(G) = 2$, and thus the graph $G$ is 2-choosable. By Lemma 2.3, the core of $G$ belongs to $T = \{K_1, C_{2m+2}, \Theta_{2,2,2m} : m \geq 1\}$. Moreover, $\Theta_{2,2,2m}$ (with $m \geq 1$) contains no odd cycle, and is therefore bipartite. Note that it has $2m+3$ vertices and $2m+4$ edges. By Lemma 2.2, we have $AT(\Theta_{2,2,2m}) \ge 3$, and hence the core of $G$ must belong to $\{K_1, C_{2m+2}\}$.

\noindent\textbf{Corollary 2.1.}  Let $ G $ be a connected graph that contains no odd cycles and has at least one edge, $|E(G)| = m$, and $|V(G)| = n$. Then $ AT(G) = 2 $ if and only if $ m \leq n $.

\noindent\textit{Proof:} We first prove the necessity. Since $ AT(G) = 2 $, we have $ \chi(G) = 2 $, and thus the graph $ G $ is bipartite. By Lemma~2.2, the conclusion follows immediately. We now prove the sufficiency. If the number of edges of $ G $ is less than or equal to the number of its vertices, and since any connected graph satisfies $ m \geq n - 1 $, it follows that either $ m = n - 1 $ or $ m = n $. When $ m = n - 1 $, $ G $ is a tree; when $ m = n $, $ G $ is either an even cycle or a unicyclic graph (with the requirement that the cycle is an even cycle). Consequently, the core of $ G $ belongs to $ \{K_1, C_{2m+2}\} $. By Lemma~2.4, we have $ AT(G) = 2 $.

\section{The main results}

\begin{theorem} For any $ l \in \mathbb{N^+} $, let $ G $ be an arbitrary graph, and $ H $ be a $l$-degenerate graph. 
When $ l = 1 $ or $ l = 2 $, the graph $ G {+}_{S} H $ is $2$-degenerate; when $ l \geq 3 $, the graph $ G {+}_{S} H $ is $l$-degenerate.
\end{theorem}

\noindent\textit{Proof:} 
We consider it from the following two cases:

\textbf{Case 1.} $l=1,2$, we remove all vertices of degree $2$ from graph $G {+}_{S} H$. First, remove the newly added vertices of degree 2 in the construction of $S(G)$ from graph $ G {+}_{S} H$. Subsequently, we obtain the union of $n$ subgraphs $H$.

Since each $H$ is either $1$-degenerate or $2$-degenerate,

we continue to remove vertices of degree less than or equal to $2$ step by step and thereby obtain an empty graph.

Therefore, the graph $ G {+}_{S} H $ is $2$-degenerate(see Figure 5 for $G = C_3,H = C_3$).

\begin{figure}[htbp]
	\centering
	\includegraphics[height=5.1cm, width=0.35\textwidth]{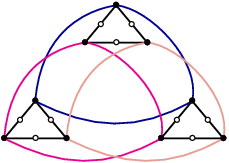}
	\caption{ $ C_3 {+}_{S} C_3$.}
\end{figure}

\textbf{Case 2.} $l\geq3$, we remove all vertices of degree less than or equal to $l$ from graph $ G {+}_{S} H $. First, remove the newly inserted degree 2 vertices in the construction of $ S(G) $  from graph $ G {+}_{S} H $.

 Subsequently, we obtain the union of $ n $ subgraphs $ H $. 
 
 Since each $ H $ is either $l$-degenerate($l\geq3$), 
 
 we continue to remove vertices of degree less than or equal to $l$ step by step and thereby obtain an empty graph.

 Therefore, the graph $ G {+}_{S} H $ is $l$-degenerate.

\noindent\textbf{Corollary 3.1.}\textit{ For any $ l \in \mathbb{N^+} $, let $G$ be an arbitrary graph, and $ H $ be a $l$-degenerate graph. Thus
\[
AT(G {+}_{S} H) \leq 
\begin{cases}
3, & \text{if } l=1,2; \\
l+1, & \text{if } l\geq3.
\end{cases}
\]
}

\noindent\textbf{Corollary 3.2.} \textit{Let $ G $ be an arbitrary graph, $H$ is a odd cycles, then $ AT(G {+}_{S} H) = 3 $.}

\noindent\textit{Proof:}
On the one hand, since graph $ H $ is a cycle, we have $ l = 2 $. According to Corollary 3.1, we know that $ AT(G {+}_{S} H) \leq 3 $. On the other hand, since graph $ H $ is an odd cycle, we have $ AT(G {+}_{S} H) \geq \chi (G {+}_{S} H)\geq \chi (H) = 3$. In summary, $ AT(G {+}_{S} H) = 3 $.

\noindent\textbf{Remark.} Therefore, the bound in Corollary~3.1 is tight when $l = 2$.

\noindent\textbf{Corollary 3.3.} \textit{For any $ n, m \in \mathbb{N} $, let $ P_n $ be a path with $ n $ vertices ($ n \geq 2) $, and $ P_m $ be a path with $ m $ vertices ($ m \geq 2) $. Then
\[
AT(P_n {+}_{S} P_m) = 
\begin{cases}
2, & \text{if } n=2,m=2; \\
3, & \text{else }.
\end{cases}
\]
}

\noindent\textit{Proof:} 
We consider it from the following two cases:

\textbf{Case 1.} $(n,m)=(2,2)$, $P_n {+}_{S} P_m$ is an even cycle, $ AT(P_n {+}_{S} P_m) = 2 $.

\textbf{Case 2.} $(n,m)\neq(2,2)$. Since graph $ P_n {+}_{S} P_m $ does not contain any odd cycles, graph $ P_n {+}_{S} P_m $ is a bipartite graph. The number of vertices is $ |V(P_n {+}_{S} P_m)| = (n + n - 1)m = 2nm - m $. The number of edges is $ |E(P_n {+}_{S} P_m)| = (n - 1)2m + n(m - 1) = 3nm - 2m -n $. By Lemma 2.2, we know that $AT(P_n {+}_{S} P_m) \geq \left\lceil \frac{3nm - 2m -n}{2nm - m} \right\rceil + 1 = 3.$ Furthermore, $P_m$ is $1$-degenerate, according to Corollary 3.1, $ AT(P_n {+}_{S} P_m) \leq 3 $. 
To sum up $ AT(P_n {+}_{S} P_m) = 3 $(see Figure 6 for $n = 5,m = 4$).

\begin{figure}[htbp]
	\centering
	\includegraphics[height=4.6cm, width=0.31\textwidth]{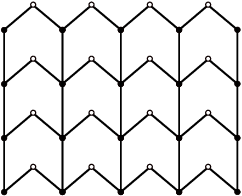}
	\caption{$ P_5 {+}_{S} P_4 $.}
\end{figure}

\noindent\textbf{Remark.} In fact, since $ AT(P_m) = 2 $ and both $ P_n $ and $ P_m $ are connected graphs, we can directly obtain $ AT(P_n {+}_{S} P_m) $ by applying the following Theorem 6.

\newpage
\noindent\textbf{Corollary 3.4.} \textit{For any $ n, m \in \mathbb{N} $, let $ C_n $ be a cycle with $ n $ vertices ($ n \geq 3) $, and $ P_m $ be a path with $ m $ vertices ($ m \geq 2) $. Then
\[
AT(C_n {+}_{S} P_m) = 3.
\]
}

\noindent\textit{Proof:} Since graph $C_n {+} _{S}P_m$ does not contain any odd cycles, graph $C_n {+} _{S}P_m$ is a bipartite graph. The number of vertices is $|V(C_n +_{S}P_m)| = 2nm$. The number of edges is $|E(C_n {+} _{S}P_m)| = 2nm + n(m - 1) = 3nm - n$. By Lemma 2.2, we know that $AT(C_n {+}_{S}P_m) \geq \left\lceil \frac{3nm - n}{2nm} \right\rceil + 1 = 3$. Furthermore, $P_m$ is $1$-degenerate, according to Corollary 3.1, $ AT(C_n {+}_{S} P_m) \leq 3 $. To sum up $ AT(C_n {+}_{S} P_m) = 3 $(see Figure 7 for $n = 4,m = 3$).

\begin{figure}[htbp]
	\centering
	\includegraphics[height=5cm, width=0.215\textwidth]{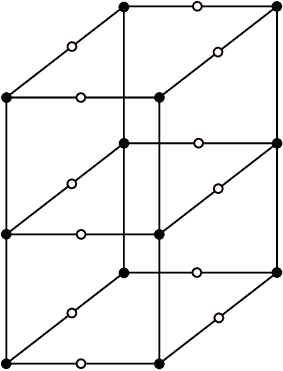}
	\caption{$ C_4 {+}_{S} P_3 $.}
\end{figure}

\noindent\textbf{Remark 1.} Since the Alon-Trasi of number Cartesian product of cycle and path is 3, $ AT(C_n {+}_{S} P_m) \leq 3. $ This way we have provided a method to find an upper bound for $AT(C_n {+}_{S} P_m)$ without using Corollary 3.1.

\noindent\textbf{Remark 2.}  In fact, since $ AT(P_m) = 2 $ and both $ C_n $ and $ P_m $ are connected graphs, we can directly obtain $ AT(C_n {+}_{S} P_m) = 3 $ by applying the following Theorem 6.

\noindent\textbf{Remark 3.} Therefore, the bound in Corollary~3.1 is tight when $l = 1$.

\noindent\textbf{Corollary 3.5.}\textit{ For any $ n, m \in \mathbb{N} $, star graph $ S_n $ is indeed a complete bipartite graph $ G[1,n] $($ n \geq 3 )$, and $ P_m $ be a path with $ m $ vertices ($ m \geq 2 )$. Then
\[
AT(S_n {+}_{S} P_m) = 3.
\]
}

\noindent\textit{Proof:} Since graph $ S_n {+}_{S} P_m $ does not contain any odd cycles,  graph $ S_n {+}_{S} P_m $ is a bipartite graph. The number of vertices is $ |V(S_n {+}_{S} P_m)| = (1 + n + n)m = 2nm + m $. The number of edges is $ |E(S_n {+}_{S} P_m)| = 2nm + (n + 1)(m - 1) = 3nm - n +m -1 $. By Lemma 2.2, we know that $AT(S_n {+}_{S} P_m) \geq \left\lceil \frac{3nm - n +m -1}{2nm + m} \right\rceil + 1 = 3.$ Furthermore, $P_m$ is $1$-degenerate, according to Corollary 3.1, $ AT(S_n {+}_{S} P_m) \leq 3 $.
To sum up $ AT(S_n {+}_{S} P_m) = 3 $(see Figure 8 for $n = 4,m = 3$).

\begin{figure}[htbp]
	\centering
	\includegraphics[height=5.6cm, width=0.248\textwidth]{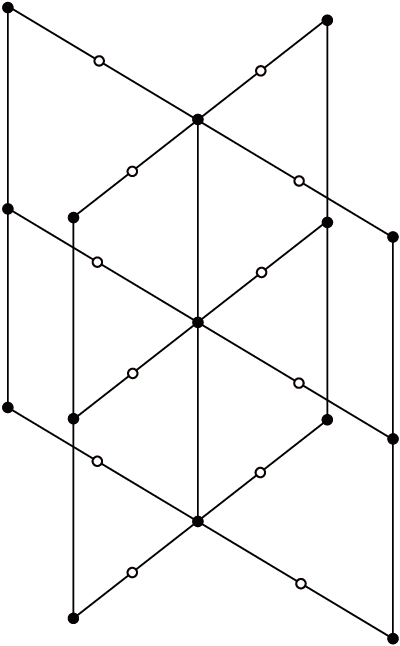}
	\caption{$ S_4 {+}_{S} P_3 $.}
\end{figure}

\noindent\textbf{Remark.} Actually, since $ AT(P_m) = 2 $ and both $ S_n $ and $ P_m $ are connected graphs, we can directly obtain $ AT(S_n {+}_{S} P_m) = 3 $ by applying the following Theorem 6.

\begin{theorem} For any $ k,l \in \mathbb{N^+} $, let $ G $ be  $k$-degenerate graph, and $ H $ be a $l$-degenerate graph. 
Then, the graph $ G {+}_{R} H $ is $(k+l)$-degenerate.
\end{theorem}

\noindent\textit{Proof:} 
We consider it from the following two cases:

\textbf{Case 1.} $ k = 1 $, graph $ G $ is a forest. Particularly, when graph $G$ is a tree. We first remove the newly created vertices of degree 2 when constructing $ R(G) $ in graph $ G {+}_{R} H $, therefore, graph $ G {+}_{R} H $ is reduced to the Cartesian product of graph $G$ and a graph $H$. Since the Cartesian product of $ G $ and $ H $ is $ (1+l) $-degenerate, by gradually removing vertices of degree less than or equal to $ (1+l) $ from graph $ G {+}_{R} H $, we can obtain an empty graph, thereby demonstrating that graph $ G {+}_{R} H $ is $ (1+l) $-degenerate(see Figure 9 for $G = T_6,H=P_3$). Moreover, since every forest is a subgraph of some tree, it follows that when the graph $ G $ is a forest, we also have $ G {+}_{R} H $ is $ (1+l) $-degenerate.

\begin{figure}[htbp]
	\centering
	\includegraphics[height=7cm, width=0.67\textwidth]{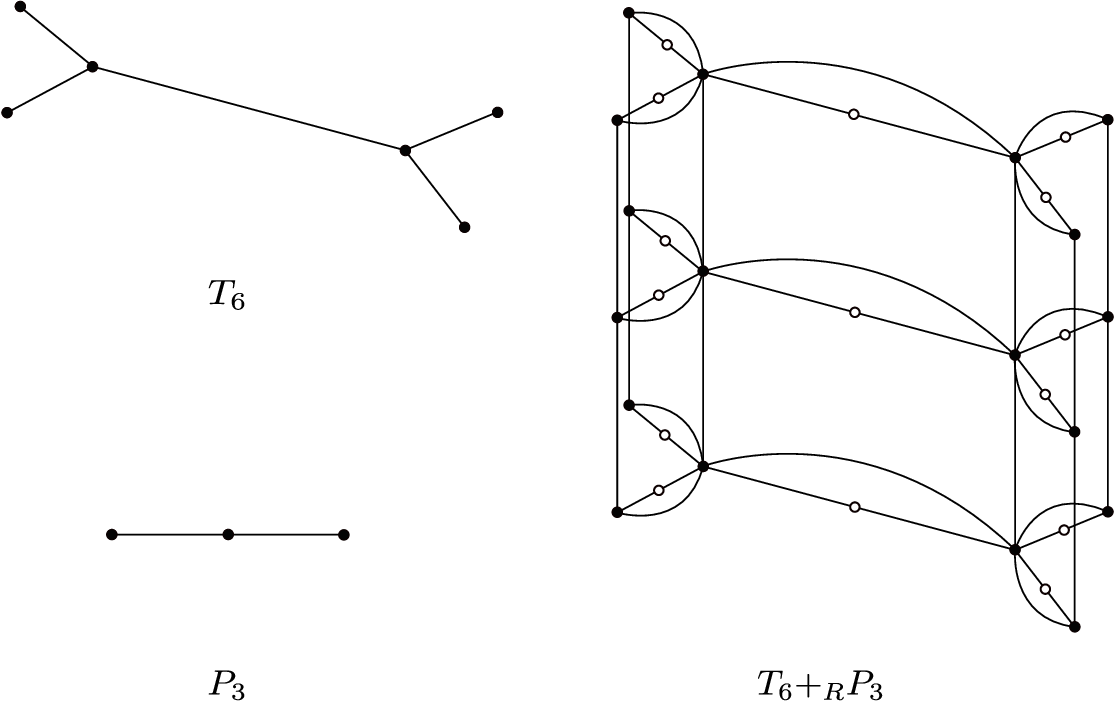}
	\caption{ $T_6$, $P_3$ and $ T_6 {+}_{R} P_3 $.}
\end{figure}

\textbf{Case 2.} $ k \geq 2 $, given that graph $ G $ is $ k $-degenerate, we first remove the newly created vertices of degree 2 when constructing $ R(G) $, therefore, graph $ R(G) $ is reduced to graph $ G $. Since graph $ G $ is $ k $-degenerate, continue to remove vertices with degrees less than or equal to $ k $, we eventually obtain an empty graph. Therefore, graph $ R(G) $ is $ k $-degenerate. Given that $ H $ possesses the property of being an $ l $-degenerate graph, according to Lemma 2.1, the Cartesian product $ R(G) \square H $ consequently exhibits $ (k+l)$-degenerate. Moreover, considering that graph $ G {+}_{R} H $ is a subgraph within the Cartesian product $ R(G) \square H $, it follows that graph $ G {+}_{R} H $ also maintains $ (k+l) $-degenerate.

\noindent\textbf{Corollary 3.6.}\textit{ For any $ k,l \in \mathbb{N^+} $, let $ G $ be a $k$-degenerate graph, and $ H $ be a $l$-degenerate graph. Then
\[
AT(G {+}_{R} H) \leq k+l+1.
\]
}

\noindent\textbf{Corollary 3.7.}\textit{ For any $ n, m \in \mathbb{N} $, let $ P_n $ be a path with $ n $ vertices ($ n \geq 2 )$, $ P_m $ be a path with $ m $ vertices ($ m \geq 2 )$. Then
\[
AT(P_n {+}_{R} P_m) = 3.
\]
}

\noindent\textit{Proof:} Since graph $ P_n {+}_{R} P_m $ does contain odd cycles, the chromatic number $ \chi(P_n {+}_{R} P_m) \geq 3$, then $ AT(P_n {+}_{R} P_m) \geq 3 $. Put it another way, $ P_n$ and $ P_m$ are $1$-degenerate, by  Corollary 3.6,  $ AT(P_n {+}_{R} P_m) \leq 3 $. We know that $ AT(P_n {+}_{R} P_m) = 3 $.

\noindent\textbf{Remark 1.} Similar to the proof of Corollary 3.7, for any $ n, m \in \mathbb{N} $, if the star graph $ S_n $ is indeed the complete bipartite graph $ G[1,n] $ ($ n \geq 3 $) and $ P_m $ is a path with $ m $ vertices ($ m \geq 2 $), then $ AT(S_n {+}_{R} P_m) = 3 $.

\noindent\textbf{Remark 2.} Similar to the proof of Corollary 3.7, for any $ n, m \in \mathbb{N} $, star graph $ S_n $ is indeed a complete bipartite graph $ G[1,n] $($ n \geq 3 )$, star graph $ S_m $ is indeed a complete bipartite graph $ G[1,m] $($ m \geq 3 )$, then $ AT(S_n {+}_{R} S_m) = 3 $.

\noindent\textbf{Remark 3.} Therefore, the inequality in Corollary~3.6 is tight.

\begin{theorem} Let the maximum degree of graph  G  be  $\Delta(G)$($(\Delta(G) > 1$), $ G $ be a $k$-degenerate graph and $ H $ be a $l$-degenerate graph. Then, the graph $ G {+}_{Q} H $ is $\max\{2\Delta(G)-2, k+l\}$-degenerate.
\end{theorem}

\noindent\textit{Proof:} When $2\Delta(G)-2 \geq k+l$, by gradually deleting vertices of degree less than or equal to $2\Delta(G)-2$ in graph $G {+}_{Q} H$, an empty graph can be obtained. Similarly, when $k+l > 2\Delta(G)-2$, by gradually deleting vertices of degree less than or equal to $l$ in graph $G {+}_{Q} H$, an empty graph can also be obtained, the graph $ G {+}_{Q} H $ is $\max\{2\Delta(G)-2, k+l\}$-degenerate(see Figure 10 for $G = S_5,H=P_3$).

\noindent\textbf{Remark 1.} When $\Delta(G)=0$, $G$ is an edgeless graph, so there is nothing to discuss.

\noindent\textbf{Remark 2.} When $\Delta(G) = 1$, the conclusion of Theorem 1 can be applied directly; hence, this case is not discussed in Theorem 3.

\begin{figure}[htbp]
	\centering
	\includegraphics[height=7.5cm, width=0.65\textwidth]{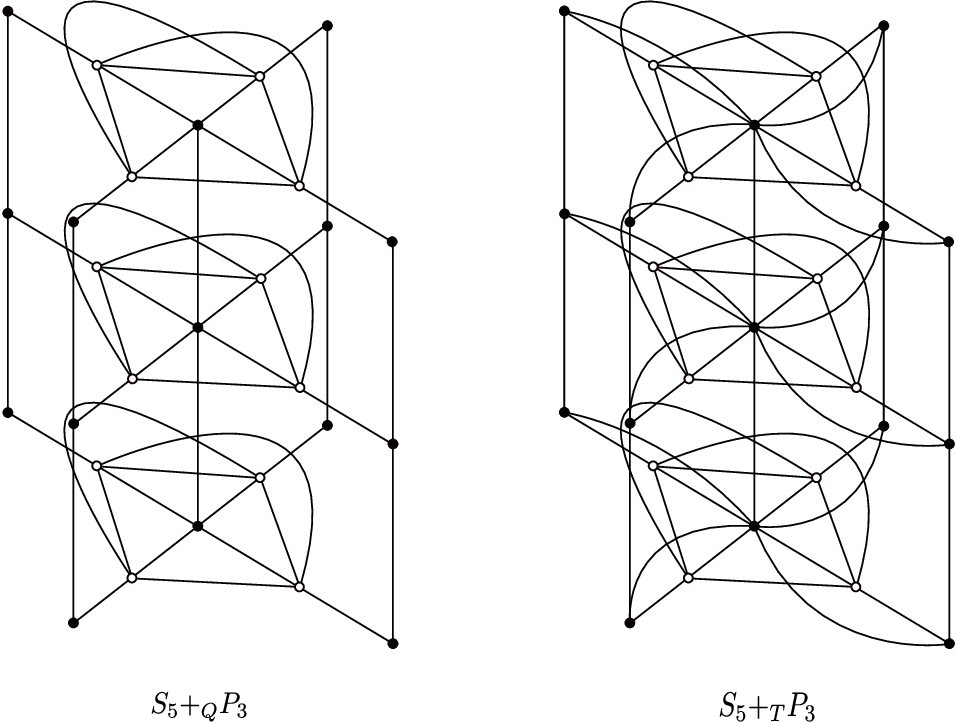}
	\caption{ $ S_5 {+}_{Q} P_3 $, $ S_5 {+}_{T} P_3 $.}
\end{figure}

\noindent\textbf{Corollary 3.8.}\textit{ Let the maximum degree of graph  G  be  $\Delta(G)$, $ G $ be a $k$-degenerate graph and $ H $ be a $l$-degenerate graph. Then
\[
AT(G {+}_{Q} H) \leq \max\{2\Delta(G)-2, k+l\}+1.
\]
}

\noindent\textbf{Corollary 3.9.}\textit{ For any $ n, m \in \mathbb{N} $, let $ P_n $ be a path with $ n $ vertices ($ n \geq 2) $, and $ P_m $ be a path with $ m $ vertices ($ m \geq 2) $. Then
\[
AT(P_n {+}_{Q} P_m) = 
\begin{cases}
2, & \text{if } n=2,m=2; \\
3, & \text{else }.
\end{cases}
\]
}

\noindent\textit{Proof:} 
We consider the problem from the following two cases:

\textbf{Case 1.} Observe the structure of the graph and by Corollary 3.1,
\[
AT(P_2 {+}_{Q} P_m) = AT(P_2 {+}_{S} P_m) = 
\begin{cases}
2, & \text{if } m=2; \\
3, & \text{else}.
\end{cases}
\]

\textbf{Case 2.} Since graph $ P_n {+}_{Q} P_m (n>2)$ does contain odd cycles, the chromatic number $ \chi(P_n {+}_{Q} P_m) \geq 3$, then $ AT(P_n {+}_{Q} P_m) \geq 3 $. On the flip side, since the maximum degree of $ P_n $ is 2, $ P_n $ and $ P_m $ is 1-degenerate, it follows from Corollary 3.8 that $ AT(P_n {+}_{Q} P_m) \leq 3$, therefore $ AT(P_n {+}_{Q} P_m) = 3. $

All in all 
\[
AT(P_n {+}_{Q} P_m) = 
\begin{cases}
2, & \text{if } n=2,m=2; \\
3, & \text{else }.
\end{cases}
\]
(see Figure 11 for $n = 5,m = 4$).

\begin{figure}[htbp]
	\centering
	\includegraphics[height=4.83cm, width=0.325\textwidth]{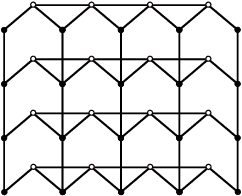}
	\caption{$ P_5 {+}_{Q} P_4 $.}
\end{figure}

\noindent\textbf{Remark.} Therefore, the inequality in Corollary~3.8 is tight.

\begin{theorem} Let the maximum degree of graph  G  be  $\Delta(G)$, $ G $ be a $k$-degenerate graph and $ H $ be a $l$-degenerate graph. Then, the graph $ G {+}_{T} H $ is $\max\{2\Delta(G), k+l\}$-degenerate.
\end{theorem}

\noindent\textit{Proof:} When $2\Delta(G) \geq k+l$, by gradually deleting vertices of degree less than or equal to $2\Delta(G)$ in graph $G {+}_{Q} H$, an empty graph can be obtained. Similarly, when $k+l > 2\Delta(G)$, by gradually deleting vertices of degree less than or equal to $l$ in graph $G {+}_{T} H$, an empty graph can also be obtained, the graph $ G {+}_{T} H $ is $\max\{2\Delta(G), k+l\}$-degenerate(see Figure 10 for $G = S_5, H = P_3$).

\noindent\textbf{Corollary 3.10.}\textit{ Let the maximum degree of graph  G  be  $\Delta(G)$, $ G $ be a $k$-degenerate graph and $ H $ be a $l$-degenerate graph. Then
\[
AT(G {+}_{T} H) \leq \max\{2\Delta(G)  , k+l\}+1.
\]
}

\noindent\textbf{Corollary 3.11.}\textit{ For any $ n, m \in \mathbb{N} $, let $ P_n $ be a path with $ n $ vertices ($ n \geq 2 ) $, and $ P_m $ be a path with $ m $ vertices ($ m \geq 2 )$. Then
\[
 AT(P_n {+}_{T} P_m) = 
\begin{cases}
3, & \text{else }; \\
4, & \text{if } n = 4 ,m \geq 5; n \geq 7 ,m = 2; n \geq 5 ,m \geq 3.
\end{cases}
\]
}

\noindent\textit{Proof:}
We consider the problem from the following five cases:

\textbf{Case 1.} $n = 2, m \geq 2$. $P_2 {+}_{T} P_m$ is identical to $P_2 {+}_{R} P_m$, according to Corollary 3.7, we know that $AT(P_2 {+}_{T} P_m) = 3 $.

\textbf{Case 2.} $n = 3, m \geq 2 $. When $ n = 3 $ and $ m = 2 $. On the one hand, we orient the edges of $P_3 {+}_{T} P_2$ as shown in Figure 12, we obtained an AT-orientation, at which point the maximum out-degree is 2, consequently $AT(P_2 {+}_{T} P_3) \leq 3$. On the other hand, $AT(P_3 {+}_{T} P_2) \geq \chi(P_3 {+}_{T} P_2) = 3$, hence $AT(P_3 {+}_{T} P_2) = 3$.

When $ n = 3 $ and $ m > 2 $. By copying the orientation of the lower layer of graph $P_3 {+}_{T} P_2$ (where $P_3 {+}_{T} P_2$ is divided into upper and lower layers), a new orientation is obtained as shown in Figure~12(see Figure 12 for $P_3 {+}_{T} P_3 $ and $P_3 {+}_{T} P_4$), we can obtain $AT(P_3 {+}_{T} P_m) \leq 3$. On the other hand, $AT(P_3 {+}_{T} P_m) \geq \chi(P_3 {+}_{T} P_m) = 3$, hence $AT(P_3 {+}_{T} P_m) = 3$.

\begin{figure}[htbp]
	\centering
	\includegraphics[height=5.5cm, width=0.68\textwidth]{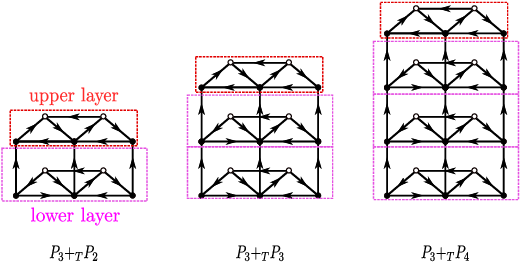}
	\caption{$ P_3 {+}_{T} P_2 $, $ P_3 {+}_{T} P_3 $ and $ P_3 {+}_{T} P_4 $.}
\end{figure}

\textbf{Case 3.} $n=4, m=2$; $n=4, m=3$; $n=4, m=4$. 
When $n=4, m=4$, the graph $ G $ is shown in Figure 13 below. The graph polynomial is given by: $f(x) = (x_1-x_2)(x_1-x_4)(x_1-x_5)(x_2-x_3)(x_2-x_5)(x_2-x_6)(x_3-x_6)(x_3-x_7)
(x_4-x_5)(x_5-x_6)(x_6-x_7)
(x_4-x_{11})(x_5-x_{12})(x_6-x_{13})(x_7-x_{14})
(x_8-x_9)(x_8-x_{11})(x_8-x_{12})(x_9-x_{10})(x_9-x_{12})(x_9-x_{13})(x_{10}-x_{13})(x_{10}-x_{14})
(x_{11}-x_{12})(x_{12}-x_{13})(x_{13}-x_{14})
(x_{11}-x_{18})(x_{12}-x_{19})(x_{13}-x_{20})(x_{14}-x_{21})
(x_{15}-x_{16})(x_{15}-x_{18})(x_{15}-x_{19})(x_{16}-x_{17})(x_{16}-x_{19})(x_{16}-x_{20})(x_{17}-x_{20})(x_{17}-x_{21})
(x_{18}-x_{19})(x_{19}-x_{20})(x_{20}-x_{21})
(x_{18}-x_{25})(x_{19}-x_{26})(x_{20}-x_{22})(x_{21}-x_{28})
(x_{22}-x_{23})(x_{22}-x_{25})(x_{22}-x_{26})(x_{23}-x_{24})(x_{23}-x_{26})(x_{23}-x_{27})(x_{24}-x_{27})(x_{24}-x_{28})
(x_{25}-x_{26})(x_{26}-x_{27})(x_{27}-x_{28})$, the coefficient of the monomial \[
x_1^2x_2^2x_3^2x_4^2x_5^2x_6^2x_7^2x_8^2x_9^2%
x_{10}^2x_{11}^2x_{12}^2x_{13}^2x_{14}^2x_{15}^2x_{16}^2%
x_{17}^2x_{18}^2x_{19}^2x_{20}^2x_{21}^2x_{22}^2x_{23}^2%
x_{24}^2x_{25}^2x_{26}^2x_{27}^2x_{28}^2
\] is $ 12 $(the Python implementation is provided in the appendix), we know that $AT(P_4 {+}_{T} P_4) \leq 3$.  Additionally, $\chi(P_4 {+}_{T} P_4) \geq 3$, then $AT(P_4 {+}_{T} P_4) = 3$.
Moreover, we note that
\[
3 = AT(P_4 {+}_{T} P_4) \geq AT(P_4 {+}_{T} P_3) \geq AT(P_4 {+}_{T} P_2) \geq \chi(P_4 {+}_{T} P_2) = 3.
\]
Hence $AT(P_4 {+}_{T} P_2) = AT(P_4 {+}_{T} P_3) = AT(P_4 {+}_{T} P_4) = 3$.

\textbf{Case 4.} $n=5, m=2$; $n=6, m=2$. Similar to the proof of Case 3, then $AT(P_5 {+}_{T} P_2) = AT(P_6 {+}_{T} P_2) = 3$. 

\textbf{Case 5.} $n = 4 ,m \geq 5$; $n \geq 7 ,m = 2$; $n \geq 5 ,m \geq 3$. On the one hand, we claim that every orientation $D$ of $ P_n {+}_{T} P_m $ such that $ d^+_D(v) \geq 3$. Assume that there exists a vertex whose out-degree satisfies $ d^+_D(v) \leq 2$, then
\[
5nm-5m-n = |E(G)| = \sum_{v \in V(G)} d^+_D(v) \leq 2|V(G)| = 2(2nm-m) = 4nm-2m,
\] 
which implies that $(n-3)(m-1)\leq3$, a contradiction. So $AT(P_n {+}_{T} P_m) \geq 4$.

On the other hand, after successively removing all vertices of degree three from graph $ P_n {+}_{T} P_m $, obtaining an empty graph, it follows that graph $ P_n {+}_{T} P_m $ is $3$-degenerate. We know that $ AT(P_n {+}_{T} P_m) \leq 4. $
Overall $AT(P_n {+}_{T} P_m)=4$(see Figure 13 for $ P_5 {+}_{T} P_4 $).

\begin{figure}[htbp]
	\centering
	\includegraphics[height=5.4cm, width=0.68\textwidth]{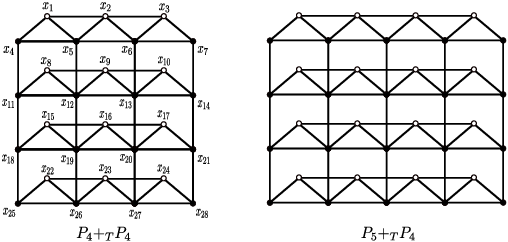}
	\caption{ $ P_4 {+}_{T} P_4 $, $ P_5 {+}_{T} P_4 $.}
\end{figure}

\begin{theorem} 
Let $K_n$ be a complete graph, then $ \text{AT}(S(K_n)) = 3 $.
\end{theorem}

\noindent\textit{Proof:} On the one hand, assign a direction to each edge of the graph $ S(K_n) $ such that the newly constructed vertex points to the original vertex, we obtain an $AT$-$orientation$ with maximum out-degree of 2, which implies that $AT(S(K_n)) \leq 3$. On the other hand, since graph $S(K_n)$ contains no odd cycles, $S(K_n)$ is a bipartite graph. According to Lemma 2.2, $AT(S(K_n)) \geq \left\lceil \frac{n(n-1)}{n+\frac{n(n-1)}{2}} \right\rceil + 1 = 3$(see Figure 14 for $K_4$ and  $S(K_4)$). In conclusion, $AT(S(K_n)) = 3$.

\begin{figure}[htbp]
	\centering
	\includegraphics[height=3.2cm, width=0.39\textwidth]{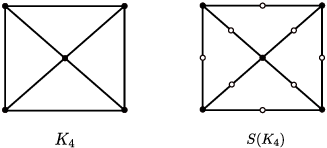}
	\caption{ $K_4$, $S(K_4)$.}
\end{figure}

\begin{theorem} 
Let $ G $ and $ H $ be any two connected graphs and $ \text{AT}(H) = k $. Then
\[
AT(G {+}_{S} H) = 
\begin{cases}
2, & \text{if } k=2 \text{ and } G=H=P_2; \\
3, & \text{if } k=2 \text{ and } G,H \neq P_2; \\
k, & \text{if } k \geq 3.
\end{cases}
\]
\end{theorem}

\noindent\textit{Proof:} 
We consider it from the following two cases:

\textbf{Case 1.} When $ k = 2 $ and both graphs $ G $ and $ H $ are the path $ P_2 $, the graph $ G {+}_{S} H $ is an even cycle, so $ AT(G {+}_{S} H) = 2 $.

\textbf{Case 2.} When $ k = 2 $ and neither graph $ G $ nor graph $ H $ can both be the path $ P_2 $. On the one hand, delete the vertices of degree 2 in the graph $ G {+}_{S} H $ of $ S(G) $, thereby obtaining some disjoint unions of graphs $ H $. According to Lemma 2.4, the structure of graph $H$ is known, we then continue deleting vertices of degree at most 2 to obtain an empty graph. Therefore, the graph  $ G {+}_{S} H $  is 2-degenerate, then $ AT(G {+}_{S} H) \leq 3 $. On the other hand, since in this case the graph $G {+}_{S} H$ contains at least two cycles, it follows from Lemma 2.4 that $AT(G {+}_{S} H) \neq 2$, i.e., $AT(G {+}_{S} H) \geq 3$. In conclusion, $ AT(G {+}_{S} H) = 3 $.

\textbf{Case 3.} When $ k \geq 3 $, on the one hand, since $ AT(H) = k $, there must exist an $AT$-$orientation$ of $ H $ such that the maximum out-degree of $ H $ is $ k-1 $. We keep this orientation unchanged in graph $ G {+}_{S} H $ and continue to orient the remaining $ S(G) $ as follows: when constructing $ S(G) $, the newly added vertices  are oriented towards the original vertices in $ G $. Thus, we obtain an $AT$-$orientation$ of graph $ G {+}_{S} H $ with a maximum out-degree of $ k-1 $, implying that $ AT(G {+}_{S} H) \leq k $. On the other hand, since $ H $ is a subgraph of $ G {+}_{S} H $, we have $ AT(G {+}_{S} H) \geq AT(H) = k $. Therefore, combining both aspects, we conclude that $ AT(G {+}_{S} H) = k $.

\section{Conclusions}

In this study, we have derived the range of $ AT(G {+}_{F} H) $ based on $ k $-degenerate. Subsequently, since connected $ S(G) $ contains no odd cycles, we have $ \chi (S(G)) = 2 $. Combining with Theorem 5, it follows that $ S(G) $ is not a chromatic-$AT$ choosable graph. Since connected $ G {+}_{S} H $ also contains no odd cycles, we have $ \chi (G {+}_{S} H) = 2 $. Combining with Theorem 6, it follows that $ G {+}_{S} H $ is not a chromatic-$AT$-choosable graph except when $ G = H = P_2 $. Additionally, according to Lemma 2.4, we obtain an equivalent characterization for $ AT(G) = 2 $. Finally, in the process of proving the theorem, we have also obtained relevant conclusions on some specific problems.

\noindent\textbf{Question.} First, for Corollary~3.1, is the result tight when $l \ge 3$? The same question applies to the bound in Corollary~3.10,  is that result tight? Second, assuming $AT(G)=k$ and $AT(H)=l$, is the Alon--Tarsi number of the $G{+}_{R}H$, $G{+}_{Q}H$, and $G{+}_{T}H$ determined solely by $k$ and $l$? Finally, are the graphs $G{+}_{R}H$, $G{+}_{Q}H$, and $G{+}_{T}H$ always chromatic--$AT$ choosable?

\section{Acknowledgement}

This work was supported in part by the National Natural Science Foundation of China (No. 12571345) and the Natural Science Foundation of Hebei Province, China (No. A2021202013). We would like to thank the anonymous referee for any helpful comments and suggestions.

\section{Appendix}
The Python implementation in Corollary 4.11, Case 3:

\begin{lstlisting}[style=PythonStyle]
import time
from collections import defaultdict

def compute_target_coefficient(edge_list, num_vars=28):
    var_edge_count = defaultdict(int)
    for i, j in edge_list:
        var_edge_count[i] += 1
        var_edge_count[j] += 1

    for var in range(1, num_vars + 1):
        if var_edge_count.get(var, 0) < 2:
            return 0

    current_counts = defaultdict(int)
    total = 0

    def backtrack(edge_idx, sign):
        nonlocal total
        if edge_idx == len(edge_list):
            for var in range(1, num_vars + 1):
                if current_counts[var] != 2:
                    return
            total += sign
            return

        i, j = edge_list[edge_idx]

        if current_counts[i] < 2:
            current_counts[i] += 1
            backtrack(edge_idx + 1, sign)
            current_counts[i] -= 1

        if current_counts[j] < 2:
            current_counts[j] += 1
            backtrack(edge_idx + 1, sign * (-1))
            current_counts[j] -= 1

    backtrack(0, 1)
    return total


edge_list = 
[
    (1, 2), (1, 4), (1, 5), (2, 3), (2, 5), (2, 6), (3, 6), (3, 7), 
    (4, 5), (5, 6), (6, 7), (4, 11), (5, 12), (6, 13), (7, 14), (8, 9), 
    (8, 11), (8, 12), (9, 10), (9, 12), (9, 13), (10, 13), (10, 14), 
    (11, 12), (12, 13), (13, 14),(11, 18), (12, 19), (13, 20), 
    (14, 21), (15, 16), (15, 18), (15, 19), (16, 17), (16, 19), 
    (16, 20), (17, 20), (17, 21),(18, 19), (19, 20), (20, 21), 
    (18, 25), (19, 26), (20, 22), (21, 28), (22, 23), (22, 25), 
    (22, 26), (23, 24), (23, 26), (23, 27), (24, 27), (24, 28), 
    (25, 26), (26, 27), (27, 28)
]

start = time.time()
coeff = compute_target_coefficient(edge_list)
end = time.time()

print(f"Coefficient of x1^2 x2^2 ... x28^2: {coeff}")
\end{lstlisting}


\begin{thebibliography}{99}
	\bibitem{BS}
	T. R. Jensen and B. Toft. \emph{Graph coloring problems}[M]. Wiley,  New York, 1995.

	\bibitem{EN}
	ELIASI M, TAERI B. Four new sums of graphs and their Wiener indices[J]. Discrete Applied Mathematics, 2008, 157(4): 794-803. 
	
	\bibitem{EN}
	Z. G. Li, Z. L. Shao, F. Petrov, and A. Gordeev. The Alon-Tarsi number of a toroidal grid[J]. European Journal of Combinatorics,111(2023), 103697.
	
	\bibitem{EN}
	Kaul H, Mudrock J A. On the Alon - Tarsi Number and Chromatic - Choosability of Cartesian Products of Graphs[J]. The Electronic Journal of Combinatorics, 2019, 26(1): \#P1.3.
	
	\bibitem{EN}
	D. Sarala, H. Y. Deng, S. K. Ayyaswamy, S. Balachandran. The Zagreb indices of graphs based on four new operations related 
to the lexicographic product[J]. Applied Mathematics and Computation 309 (2017) 156–169.

	\bibitem{EN}
	X. D. Zhu and R. Balakrishnan. \emph{Combinatorial Nullstellensatz With Applications to Graph Colouring}[M]. Chapman and Hall/CRC, 2021.

	\bibitem{EN}
	N. Alon and M. Tarsi. Colorings and orientations of graphs[J]. Combinatorica, 12(2) (1992): 125 - 134.
	
	\bibitem{EN}
	Eric Culver and Stephen G. Hartke. Relation between the correspondence chromatic number and the Alon-Tarsi number[J]. Discrete Mathematics, 346.6 (2023): 113347.
	
	\bibitem{EN}
	Z. S. Liu and Mai Tumurzhuo Maisidike. Linear Arboricity of the Product Graph of $1$-Degenerate Graphs[J]. Shandong University, vol,  1671-9352(2025)02-0051-12.


	\bibitem{EN}
	N. Eaton and T. Hull. Defective List Colorings of Planar Graphs[J].  \emph{Bull. inst. combin. appl}, 1999, 25(25):79$-$87.


	\bibitem{EN}
     AWAIS, H. M.; JAVAID, Muhammad; ALI, Akbar. First General Zagreb Index of Generalized $F$-sum Graphs[J]. Discrete Dynamics in Nature and Society, 2020, 2020.1: 2954975.

	\bibitem{EN}
Erdos P, Rubin A L, Taylor H. Choosability in graphs[J]. Congr. Numer, 1979, 26(4): 125-157.
\end{thebibliography}
\end{document}